\documentclass[a4paper]{amsart}
\usepackage{amsmath, amsfonts,amsthm,amssymb,amscd,array,enumitem,verbatim, graphicx,color,multirow,booktabs,tikz,adjustbox,setspace}
\usepackage{chngcntr}
\counterwithin{table}{section}
\def\classification#1{\def\@class{#1}}
\classification{\null}
\usepackage{framed, calc}
\usepackage[margin=2cm]{geometry}

\usepackage{url}

\newtheorem{prop}{Proposition}[section]
\newtheorem{thm}[prop]{Theorem}
\newtheorem{question}[prop]{Question}
\newtheorem{cor}[prop]{Corollary}
\newtheorem{lem}[prop]{Lemma}
\theoremstyle{definition}
\numberwithin{equation}{section}

\newenvironment{claim}[1]{\par\noindent\textbf{Claim.}\space#1}{}
\newenvironment{claimproof}[1]{\par\noindent\textit{Proof of Claim.}\space#1}{}

\setlist[enumerate]{leftmargin=20pt,itemsep=0pt,topsep=0pt}
\setlist[enumerate,1]{label=\textup{(\roman*)}}

\DeclareMathOperator{\Symm}{Sym}

\DeclareMathOperator{\DP}{\textup{DP}}
\DeclareMathOperator{\PR}{\textup{P}}
\DeclareMathOperator{\PP}{\textup{PP}}
\DeclareMathOperator{\NORM}{\textup{N}}
\DeclareMathOperator{\CENT}{\textup{C}}

\newcommand\smallstrut{\rule{0pt}{7pt}}
\newcommand\smallerstrut{\rule{0pt}{6.5pt}}

\begin{document}
\title{Nilpotent covers of symmetric and alternating groups}

\author{Nick Gill}
\address{Department of Mathematics, University of South Wales, Treforest, CF37 1DL, United Kingdom}
\email{nick.gill@southwales.ac.uk}

\author{Ngwava Arphaxad Kimeu}
\address{P.O.BOX 116--90100, Machakos,Kenya; Moi University P.O.B0X 3900--30100, Eldoret, Kenya}
\email{kimeungwava70@gmail.com}

\author{Ian Short}
\address{School of Mathematics and Statistics, The Open University, Milton Keynes, MK7 6AA, United Kingdom}
\email{ian.short@open.ac.uk}

\begin{abstract}
We prove that the symmetric group $S_n$ has a unique minimal cover $\mathcal{M}$ by maximal nilpotent subgroups, and we obtain an explicit and easily computed formula for the order of $\mathcal{M}$. In addition, we prove that the order of $\mathcal{M}$ is equal to the order of a maximal non-nilpotent subset of $S_n$. This cover $\mathcal{M}$ has attractive properties; for instance, it is a normal cover, and the number of conjugacy classes of subgroups in the cover is equal to the number of partitions of $n$ into distinct positive integers. 

We show that these results contrast with those for the alternating group $A_n$. In particular, we prove that, for all but finitely many values of $n$, no minimal cover of $A_n$ by maximal nilpotent subgroups is a normal cover and the order of a minimal cover of $A_n$ by maximal nilpotent subgroups is strictly greater than the order of a maximal non-nilpotent subset of $A_n$. 
\end{abstract}


\maketitle

\section{Introduction}

The principal objective of this paper is to determine the least number of nilpotent subgroups of the symmetric group on $n$ letters $S_n$ that are necessary to cover $S_n$. We establish that there is a unique minimal collection of maximal nilpotent subgroups that cover $S_n$, and the order of this collection can be computed easily from a list of the partitions of $n$ into distinct positive integers. Furthermore, we prove that the order of the collection is equal to the order of a maximal non-nilpotent subset of $S_n$. 

To explain our results in more detail, consider a finite group $G$. A \emph{nilpotent cover} of $G$ is a finite family $\mathcal{M}$ of nilpotent subgroups of $G$ for which 
\[
G = \bigcup_{H\in\mathcal{M}} H.
\]
A nilpotent cover $\mathcal{M}$ of $G$ is said to be \emph{minimal} if no other nilpotent cover of $G$ has fewer members. Let $\Sigma_N(G)$ denote the size of a minimal nilpotent cover of $G$, provided such a cover exists.

Of particular interest to us are nilpotent covers that are invariant under conjugation. A nilpotent cover $\mathcal{M}$ of $G$ is \emph{normal} if whenever $H\in \mathcal{M}$ and $g\in G$ we have $g^{-1}Hg\in\mathcal{M}$. We seek to ascertain whether a minimal nilpotent cover of $G$ can be found that is normal. 

Each normal nilpotent cover of $G$ can be partitioned into conjugacy classes of subgroups. Let $\Gamma_N(G)$ denote the least number of such conjugacy classes, among all the normal nilpotent covers of $G$.

There is a parallel notion to that of a nilpotent cover: a \emph{non-nilpotent subset} of $G$ is a subset $X$ of $G$ such that for any two distinct elements $x$ and $y$ of $X$, the subgroup $\langle x,y \rangle$ generated by $x$ and $y$ is not nilpotent. A non-nilpotent subset of $G$ is said to be \emph{maximal} if no other non-nilpotent subset of $G$ contains more elements. Let $\sigma_N(G)$ denote the  size of a maximal non-nilpotent subset of $G$. A straightforward consequence of the pigeon-hole principle is that
\(
\sigma_N(G)\leq \Sigma_N(G),
\)
provided $G$ has a nilpotent cover.

In this paper we calculate the quantities $\Gamma_N(S_n)$,  $\Sigma_N(S_n)$ and $\sigma_N(S_n)$ and prove that the latter two quantities coincide. We use the concept of a \emph{distinct partition} of a positive integer $n$, which is a set $T=\{t_1,t_2,\dots, t_k\}$ where $t_1,t_2,\dots, t_k$ are distinct positive integers and $n=t_1+t_2+\dots+t_k$. 

In Section~\ref{s: section2} we will prove that if the cycle type of an element $g$ of $S_n$ is a distinct partition of $n$, then $g$ lies within a unique maximal nilpotent subgroup of $S_n$ (Proposition~\ref{p: distinct partition}). We denote by $\mathcal{M}$ the collection of all maximal nilpotent subgroups that arise in this way.

\begin{thm}\label{thm1}
The cover $\mathcal{M}$ is the unique minimal cover of $S_n$ by maximal nilpotent subgroups.
\end{thm}

The cover $\mathcal{M}$ is by definition a normal cover, so we obtain the following corollary of Proposition~\ref{p: distinct partition} and Theorem~\ref{thm1}.

\begin{cor}\label{cor1}
The cover $\mathcal{M}$ is a normal nilpotent cover of $S_n$ and $\Gamma_N(S_n)$ is equal to the number of partitions of $n$ into distinct positive integers.
\end{cor}

Moreover, using Proposition~\ref{p: distinct partition} we can see that $\mathcal{M}$ is the unique normal nilpotent cover of $S_n$ containing $\Gamma_N(S_n)$ conjugacy classes of maximal nilpotent subgroups of $S_n$.

The problem of calculating the number of distinct partitions of a positive integer $n$ is an old one. Values for small $n$ can be found at the OEIS \cite{oeis} along with a wealth of information about this problem. One well-known fact about distinct partitions (due to Euler) is that the number of distinct partitions of $n$ is equal to the number of partitions of $n$ into odd positive integers.

A further corollary of Theorem~\ref{thm1} gives an explicit formula for $\Sigma_N(S_n)$ and $\sigma_N(S_n)$.

\begin{cor}\label{cor2}
We have
\[
\Sigma_N(S_n)=
\sigma_N(S_n)=
\sum\limits_{T\in \DP(n)} 
\left(
\frac{n!}{\prod\limits_{t\in T}\prod\limits_{i=1}^\ell (p_i-1)\smallstrut^{a_i}p_i^{e_i}}\right)\!,
\]
where, in the final product, $t=p_1^{a_1}p_2^{a_2}\dotsb p_{\ell}^{a_\ell}$ is the prime factorisation of $t$ and $e_i=(p_i^{a_i}-1)/(p_i-1)$, for~$i=1,2,\dots,\ell$.
\end{cor}
	
The product $\prod\limits_{i=1}^\ell(p_i-1)\smallstrut^{a_i}p_i^{e_i}$ in Corollary~\ref{cor2} is considered to take the value 1 if $t=1$.	

Let $\DP(n)$ denote the set of all distinct partitions of $n$.  Table~\ref{table1} displays the first few values of $\DP(n)$ and~$\Sigma_N(S_n)$.

\begin{table}[ht]
\begin{tabular}{ p{1.5cm} p{9.5cm} p{2cm} }
$n$ & $\DP(n)$ & $\Sigma_N(S_n)$ \\
\hline\\[-6pt]
2 & $\{2\}$ & 1\\
3 & $\{1,2\},\{3\}$ & 4 \\
4 & $\{1,3\},\{4\}$ & 7 \\
5 & $\{1,4\},\{2,3\},\{5\}$ & 31 \\
6 & $\{1,2,3\}, \{1,5\},\{2,4\},\{6\}$ & 201 \\
7 & $\{1,2,4\}, \{1,6\},\{2,5\},\{3,4\},\{7\}$ & 1086 \\
8 & $\{1,2,5\},\{1,3,4\},\{1,7\},\{2,6\},\{3,5\},\{8\}$ & 5139 \\
9 & $\{1,2,6\},\{1,3,5\},\{2,3,4\},\{1,8\},\{2,7\},\{3,6\},\{4,5\},\{9\}$ & 37507 \\[8pt]
\end{tabular}
\caption{Values of $\DP(n)$ and $\Sigma_N(S_n)$, for $n=2,3,\dots,9$.}
\label{table1}
\end{table}

In Section~\ref{section nilpotent covers} we contrast these results for the symmetric group with properties of nilpotent covers of alternating groups. Then in Section~\ref{section other covers} we consider nilpotent covers and abelian covers of other almost simple groups.

\section{Nilpotent covers of symmetric groups}\label{s: section2}

Here we prove Theorem~\ref{thm1} and Corollary~\ref{cor2}. In the course of the proofs we use a variety of well known properties of  permutation groups and nilpotent groups. In particular, we use the fact that a finite nilpotent group is a direct product of its Sylow subgroups. A consequence of this observation is that elements of coprime order in a finite nilpotent  group commute. Another fact we use is that if $P$ is a Sylow $p$-subgroup of $S_n$, for some prime $p$, and $n=a_0+a_1p+\dotsb +a_kp^k$ is the base-$p$ expansion of $n$, then $P$ has $a_i$ orbits on $\{1,2,\dots,n\}$ of size $p^i$, for $i=0,1,\dots,k$ (and this accounts for all orbits of $P$). For example, if $p=3$ and $n=16$, then $P$ has orbits of sizes $1,3,3,9$.	

We make use of the following notation and terminology. Given a subset $X$ of $\{1,2,\dots,n\}$, we denote by $\Symm(X)$ the full group of symmetries of $X$ within $S_n$. That is, $\Symm(X)$ is the pointwise stabilizer of the complement of $X$. Also, we refer to the `orbits of a permutation $g$' as a shorthand for the orbits of the \emph{cyclic group} generated by~$g$. 

For an integer $k$ and a prime $p$, we define $|k|_p$ to be the largest power of $p$ that is a factor of $k$, and we define $|k|_{p'}=|k|/|k|_p$. 

\begin{lem}\label{l: intransitive dp}
Let $O_1,O_2,\dots,O_k$	be the orbits in $\{1,2,\dots,n\}$ of an element $g$ of $S_n$, and suppose that the orders $t_1,t_2,\dots,t_k$ of these orbits form a distinct partition of $n$. Any nilpotent subgroup $N$ of $S_n$ containing $g$ satisfies
\[
N \leq \Symm(O_1)\times \Symm(O_2) \times \dots \times \Symm(O_k).
\]
\end{lem}

Note that the supposition of nilpotency is essential here: even replacing it with solvability would not work, as the example $g=(2,3)(4,5,6)\in S_3\wr S_2<S_6$ makes clear.

\begin{proof}
Let $M=\Symm(O_1)\times \Symm(O_2) \times \dots \times \Symm(O_k)$, a subgroup of $S_n$. Since $N$ is a direct product of its Sylow subgroups, it is enough to prove that any Sylow $p$-subgroup $P$ of $N$ is contained in $M$. Let $h\in P$. We will prove that $O_i^h=O_i$, for $i=1,2,\dots, k$, from which it follows that $h\in M$. 

Let $m$ denote the order of $g$ and let $d=|m|_{p'}$ and $f=|m|_{p}$ (so $m=df$). We define $g_1=g^{f}$, which has order $d$. The orbits of $g_1$ are a refinement of those of $g$. More specifically, for each index $i$, there are $|t_i|_p$ orbits of $g_1$ within $O_i$, each of size $|t_i|_{p'}$.

Since $g_1$ and $h$ have coprime orders and lie in a nilpotent group they must commute. Consequently, for any $x\in\{1,2,\dots,n\}$, the orbit of $x$ under $g_1$ has the same size as the orbit of $x^{h}$ under $g_1$. It follows that if $x\in O_i$ and $x^{h}\in O_j$, then $|t_i|_{p'}=|t_j|_{p'}$. 
	
Consider a complete set of orbits $O_i$ for which the values $|t_i|_{p'}$ are all equal; after relabelling we can assume that $O_1,O_2,\dots, O_\ell$ is such a set. The preceding argument shows that $h$ preserves the union $O_1\cup O_2\cup\dots \cup O_\ell$. For the remainder of the argument, we focus on the restriction of $g$ and $h$ to this union. To simplify notation, we assume that this union is in fact the full set $\{1,2,\dots,n\}$.

As before, we let $m$ denote the order of $g$ and let $d=|m|_{p'}$, $f=|m|_{p}$ and $g_1=g^{f}$, of order $d$. Since each of the orbits of $g_1$ has the same size, we see that $g_1$ is a product of $n/d$ disjoint $d$-cycles. Thus $g_1$ has $n/d$ orbits, each of which comprises the $d$ entries of some $d$-cycle of $g_1$.

Any element of $S_n$ that commutes with $g_1$ must permute these $n/d$ orbits. In this way we obtain a homomorphism $\theta$ from the centralizer $\CENT_{S_n}(g_1)$ of $g_1$ in $S_n$ to $S_{n/d}$. The Sylow $p$-subgroup $P$ is contained in $\CENT_{S_n}(g_1)$, so we can define $P'=\theta(P)$, a $p$-subgroup of $S_{n/d}$.

Now define $g_2=g^d$. This has order $f$, a power of $p$, so $g_2\in P$. Let us consider how the cyclic group $\langle g_2\rangle$ acts on the orbits of $g_1$. For any index $i\in\{1,2,\dots,\ell\}$, we know that $g_2$ fixes $O_i$. Choose $x,y\in O_i$; then $y=g^r(x)$ for some integer $r$. Since $d$ and $f$ are coprime we can find integers $a$ and $b$ such that $r=ad+bf$, in which case $y=g^{ad+bf}(x)$, so $g_2^a(x)=g_1^{-b}(y)$. It follows that $g_2^a(x)$ and $y$ are in the same orbit of $g_1$. Therefore any orbit of $g_1$ in $O_i$ can be mapped to any other under the action of $\langle g_2\rangle$. Consequently, $\theta(g_2)$ is a product of disjoint cycles of orders $t_1/d,t_2/d,\dots,t_\ell/d$. 

Each of the integers $t_1/d,t_2/d,\dots,t_\ell/d$ is a power of $p$, and they are distinct, since the integers $t_1,t_2,\dots,t_\ell$ are distinct. Thus $n/d=t_1/d + t_2/d +\dots + t_\ell/d$ is the $p$-ary decomposition of $n/d$. Consider next the action of the $p$-subgroup $P'$ of $S_{n/d}$. Using the observation about orbits of Sylow $p$-subgroups stated before the lemma we see that the orbits of $P'$ must be exactly those of $\theta(g_2)$. In particular, since $\theta(h)\in P'$ we see that, for any $i\in\{1,2,\dots,\ell\}$, the permutation $h$ maps any orbit of $g_1$ within $O_i$ to another orbit of $g_1$ in $O_i$. Hence $h$ fixes $O_i$, as required. 
\end{proof}

Let $p$ be a prime, $a$ a positive integer, and $q=p^a$. The next lemma uses the known result that all $q$-cycles in a Sylow $p$-subgroup $P$ of $S_q$ are conjugate in the  normalizer $\NORM_{S_q}(P)$ of $P$ in $S_q$.

\begin{lem}\label{l: Sylow 2}
Let $g$ be a $q$-cycle in $S_q$, where $q=p^a$, a prime power. There is a unique Sylow $p$-subgroup of $S_q$ that contains $g$.
\end{lem}
\begin{proof}
Let $P$ be a Sylow $p$-subgroup of $S_q$ that contains $g$, and suppose that $g$ also belongs to another Sylow $p$-subgroup $h^{-1}Ph$, for some $h\in S_q$. Then $hgh^{-1}$ is a $q$-cycle in $P$, so there exists $k\in \NORM_{S_q}(P)$ with $hgh^{-1}=k^{-1}gk$. Consequently, $(kh)g=g(kh)$, so $kh$ belongs to the centralizer of $g$ in $S_q$. Now, the centralizer of $g$ is the cyclic group generated by $g$, so  $h\in\NORM_{S_q}(P)$. Hence $h^{-1}Ph=P$. Thus $g$ is contained in a unique Sylow $p$-subgroup of $S_q$. 
\end{proof}

Next we introduce some concepts about permutation groups and Sylow subgroups.

Suppose that $G$ is a subgroup of $S_m$ and $H$ is a subgroup of $S_n$. Then $G\times H$ acts faithfully on $\{1,2,\dots,m\}\times\{1,2,\dots,n\}$ by the formula $(g,h)\colon(x,y)\longmapsto (x\smallerstrut^g,y^h)$. By choosing some identification of $\{1,2,\dots,m\}\times\{1,2,\dots,n\}$ with $\{1,2,\dots,mn\}$ we obtain a subgroup of $S_{mn}$, which we denote by $G\dot{\times}H$ (the freedom to choose an identification implies that this group is defined only up to conjugation in $S_{mn}$). 

In the same way we can take subgroups $G_1, G_2,\dots, G_k$ of the symmetric groups $S_{n_1},S_{n_2},\dots,S_{n_k}$, in order, and define the product $G_1\dot{\times} G_2\dot{\times} \dotsb \dot{\times} G_k$, a subgroup of $S_{n}$, where $n=n_1 n_2 \dotsb  n_k$. Observe that the operation $\dot{\times}$ is both associative and commutative. Observe also that if $G_1,G_2,\dots, G_k$ are transitive subgroups of $S_{n_1}, S_{n_2},\dots, S_{n_k}$, then $G_1\dot{\times} G_2\dot{\times} \dotsb \dot{\times} G_k$ is a transitive subgroup of $S_{n}$.

For a positive integer $t$, we write $t=p_1^{a_1}p_2^{a_2} \dotsb p_\ell^{a_\ell}$, where $p_1<p_2<\dots<p_\ell$ are primes and $a_1,a_2,\dots,a_\ell$ are positive integers. Let $q_i=p_i^{a_i}$, for $i=1,2,\dots, \ell$. We define $\PR(t)=[p_1,p_2,\dots, p_\ell]$ and $\PP(t)=[q_1,q_2,\dots, q_\ell]$, the lists of primes and the corresponding prime powers in the prime factorisation of $t$, respectively, each written in increasing order.

Let $P_{p_i,q_i}$ be a Sylow $p_i$-subgroup of $S_{q_i}$, for $i=1,2,\dots,\ell$. Since $t=q_1q_2\dotsb q_\ell$ we see that 
\[
P_{p_1, q_1}\dot{\times} P_{p_2,q_2}\dot{\times} \dotsb \dot{\times} P_{p_\ell, q_\ell}
\]
is a subgroup of $S_t$. Indeed, it is a transitive subgroup of $S_t$, because $P_{p_i,q_i}$ is a transitive subgroup of $S_{q_i}$.

The following lemma generalises Lemma~\ref{l: Sylow 2}.

\begin{lem}\label{l: transitive dp}
Let $g$ be a $t$-cycle in $S_t$. Let $\PR(t)=[p_1,p_2,\dots, p_\ell]$ and $\PP(t)=[q_1,q_2,\dots, q_\ell]$. Then $g$ lies in a unique maximal nilpotent subgroup $N$ of $S_t$ and
\[
 N=P_{p_1, q_1}\dot{\times} P_{p_2,q_2}\dot{\times} \dotsb \dot{\times} P_{p_\ell, q_\ell},
\]
for some Sylow $p_i$-subgroups $P_{p_i,q_i}$ of $S_{q_i}$, $i=1,2,\dots,\ell$.
\end{lem}

\begin{proof}
By conjugating, we can assume that $g=(1,2,\dots,t)$. Let $N$ be a nilpotent subgroup of $S_t$ that contains $g$. Then $N$ is the direct product of its Sylow $p_i$-subgroups $P_i$, for $i=1,2,\dots,\ell$.
	
We identify the permutation set $\{1,2,\dots,t\}$ with $\mathbb{Z}/t\mathbb{Z}$ by sending $x$ to $[x]_t$, the congruence class of integers congruent to $x$ modulo $t$. We identity $\mathbb{Z}/t\mathbb{Z}$ with $\mathbb{Z}/q_1\mathbb{Z}\times\mathbb{Z}/q_2\mathbb{Z}\times\dots \times \mathbb{Z}/q_\ell\mathbb{Z}$ by sending $[x]_t$ to $([x]_{q_1},[x]_{q_2},\dots,[x]_{q_\ell})$. On $\mathbb{Z}/t\mathbb{Z}$ the action of $g$ is given by $[x]_t\longmapsto[x+1]_t$. 

Let $N_1=P_2P_3\dotsb P_\ell$ and choose $h\in N_1$. Consider any element $x=([x_1]_{q_1},[x_2]_{q_2},\dots,[x_\ell]_{q_\ell})$ of $\mathbb{Z}/q_1\mathbb{Z}\times\mathbb{Z}/q_2\mathbb{Z}\times\dots \times \mathbb{Z}/q_\ell\mathbb{Z}$ and define
\[
([y_1]_{q_1},[y_2]_{q_2},\dots,[y_\ell]_{q_\ell})=h([x_1]_{q_1},[x_2]_{q_2},\dots,[x_\ell]_{q_\ell}).
\]
We will prove that $[y_1]_{q_1}=[x_1]_{q_1}$.

Define $g_1=g^{t/q_1}$, which has order $q_1$, so it commutes with $h$. Hence $hg_1^k(x)=g_1^kh(x)$, for any integer $k$. Evaluating each side of this equation we obtain
\[
h([x_1+kt/q_1]_{q_1},[x_2]_{q_2},\dots,[x_\ell]_{q_\ell})=([y_1+kt/q_1]_{q_1},[y_2]_{q_2},\dots,[y_\ell]_{q_\ell}).
\]
Now define $g_2=g^{q_1}$, which has order $t/q_1$, so $g_2\in N_1$. We can choose an integer $m$ such that $[y_i+mq_1]_{q_i}=[x_i]_{q_i}$, for $i=2,3,\dots,\ell$. It follows that 
\[
g_2^mh([x_1+kt/q_1]_{q_1},[x_2]_{q_2},\dots,[x_\ell]_{q_\ell})=([y_1+kt/q_1]_{q_1},[x_2]_{q_2},\dots,[x_\ell]_{q_\ell}),
\]
for any integer $k$. We obtain an action of $g_2^mh$ on $\mathbb{Z}/q_1\mathbb{Z}$. However, the order of $g_2^mh$ is coprime to $q_1$, so $g_2^mh$ is the identity permutation. Hence $[x_1]_{q_1}=[y_1]_{q_1}$, as required.

A similar argument holds with the $i$th component instead of the first component. Thus, if $h\in P_j$, then $h$ fixes each component of $\mathbb{Z}/q_1\mathbb{Z}\times\mathbb{Z}/q_2\mathbb{Z}\times\dots \times \mathbb{Z}/q_\ell\mathbb{Z}$ other than the $j$th component. In this way we can identify $P_j$ with a subgroup of $\Symm(\mathbb{Z}/q_j\mathbb{Z})$, and $N=P_1\dot{\times}P_2\dot{\times}\dotsb\dot{\times} P_\ell$. Now, $g^{n/q_j}\in P_j$ and it is a $q_j$-cycle in $\Symm(\mathbb{Z}/q_j\mathbb{Z})$. By Lemma~\ref{l: Sylow 2} there is a unique Sylow $p_j$-subgroup $Q_j$ of $S_{q_j}$ that contains $g^{n/q_j}$. Taking the product 
\(
Q_1\dot{\times}Q_2\dot{\times}\dotsb\dot{\times}Q_\ell
\)
of all such groups we obtain a maximal nilpotent group containing $g$ (and $N$), uniquely specified by the subgroups $Q_j$, as required.
\end{proof}

The following proposition is an immediate consequence of Lemmas~\ref{l: intransitive dp} and~\ref{l: transitive dp}.

\begin{prop}\label{p: distinct partition}
Let $O_1,O_2,\dots,O_k$	be the orbits in $\{1,2,\dots,n\}$ of an element $g$ of $S_n$, and suppose that the orders $t_1,t_2,\dots,t_k$ of these orbits form a distinct partition of $n$. Then there is a unique maximal nilpotent subgroup $N$ of $S_n$ that contains $g$. Furthermore, there are subgroups $N_i$ of $\Symm(O_i)$, for $i=1,2,\dots,k$, with
\[
N  = N_{1}\times N_{2} \times \dots \times N_{k},
\]
where, for $i=1,2,\dots,k$,  we write $\PR(t_i)=[p_{j_1},p_{j_2},\dots, p_{j_{\ell_i}}]$ and $\PP(t_i)=[q_{j_1},q_{j_2},\dots, q_{j_{\ell_i}}]$, and we have
\[
 N_{i}=P_{p_{j_1}, q_{j_1}}\dot{\times} P_{p_{j_2},q_{j_2}}\dot{\times} \dotsb \dot{\times} P_{p_{j_{\ell_i}}, q_{j_{\ell_i}}},
\]
for some Sylow $p_{j_s}$-subgroups $P_{p_{j_s},q_{j_s}}$ of $S_{q_{j_s}}$, $s=1,2,\dots,\ell_i$.
\end{prop}

Let $g$ be an element of a subgroup $G$ of $S_m$ and let $h$ be an element of a subgroup $H$ of $S_n$. Suppose that $g$ can be expressed as a product of disjoint cycles of lengths $\lambda_1,\lambda_2,\dots,\lambda_k$ (not necessarily distinct) and $h$ can be expressed as a product of disjoint cycles of lengths $\mu_1,\mu_2,\dots,\mu_\ell$.  Then the permutation $(g,h)\colon(x,y)\longmapsto (x\smallerstrut^g,y^h)$ acting on $\{1,2,\dots,m\}\times\{1,2,\dots,n\}$ can be expressed as a product of disjoint cycles of lengths $\lambda_i\mu_j$, for $i=1,2,\dots,k$, $j=1,2,\dots,\ell$. In this manner we can determine the cycle types of all members of $G\dot{\times}H$ from those of $G$ and $H$. This observation is used in the proof of Proposition~\ref{p: cover}, to follow.

Now, let $T$ be a distinct partition of the positive integer $n$; that is, $T\in \DP(n)$. Given an element $g$ of $S_n$ with cycle type $T$ we let $N_g$ denote the unique maximal nilpotent subgroup of $S_n$ containing $g$. Then we define
\[
 N(T) = \{ h^{-1}N_gh : h\in S_n\}.
\]
Clearly, this definition does not depend on the choice of permutation $g$ of cycle type $T$.

We are now able to state our final result. 

\begin{prop}\label{p: cover}
We have \(S_n\,=\bigcup\limits_{T\in \DP(n)}N(T).\)
\end{prop}

\begin{proof}
Observe that if $q$ is a power of a prime $p$, then the group $P_{p,q}$ contains a $q$-cycle. This implies, in the notation of Proposition~\ref{p: distinct partition}, that the group $N_{i}$ contains
\[
 C_{q_{j_1}} \dot{\times} C_{q_{j_2}} \dot{\times} \cdots \dot{\times} C_{q_{j_k}},
\]
where each group $C_{q_{j_s}}$ is a cyclic permutation group generated by a $q_{j_s}$-cycle. Hence $N_i$ contains a $t_i$-cycle. It follows, in turn, that the group $N$ contains the  product  of $k$ disjoint cycles of lengths $t_1,t_2,\dots, t_k$. 

We define
\[
\Omega\,=\bigcup\limits_{T\in \DP(n)}N(T).
\]
 Choose an  an arbitrary element $h$ of $S_n$. We wish to show that $h\in \Omega$, and since $\Omega$ is a union of conjugacy classes of $S_n$,  it is enough to show that $\Omega$ contains an element conjugate to $h$. Let $R=\lambda_1^{a_1}\lambda_2^{a_2}\dotsb \lambda_k^{a_k}$ be the cycle type of $h$, using the usual partition notation: $h$ is  a product of $a_i$ disjoint $\lambda_i$-cycles, for $i=1,2,\dots, k$, and $a_1\lambda_1+a_2\lambda_2+\dots +a_k\lambda_k=n$.

Now we describe a process for amalgamating parts of the partition $R$ ``in pairs'': for $i=1,2,\dots, k$, whenever $a_i>1$ we replace $\lambda_i^{a_i}$ by
\[
 \begin{cases}
  (2\lambda_i)^{a_i/2}, & \textrm{ if $a_i$ is even},\\
  (2\lambda_i)^{(a_i-1)/2}\lambda_i, & \textrm{ if $a_i$ is odd}.  
 \end{cases}
\]
We repeat this process until we have a distinct partition $T$. For example, if $R=2^23^24^36^18^116^1$, a partition of $52$, then
\[
 2^23^24^36^18^116^1 \, \longrightarrow \, 4^16^18^14^16^18^116^1=4^26^28^216^1 \, \longrightarrow \, 8^112^116^116^1=8^112^116^2\, \longrightarrow \, 8^112^132^1.
\]
We claim that the groups in $N(T)$ contain elements of cycle type $R$. To see this, observe that, given a part $t$ of $T$, we can work backwards through the algorithm from $T$ to $R$ to obtain a list $R_t$ of parts of $T$ whose sum is $t$ (not necessarily unique). Each element of $R_t$ is a factor of $t$ with quotient a power of 2. By working through the parts of $T$ one by one, we can choose the lists $R_t$, for $t\in T$, to be a partition of $R$. For instance, using the example above with $T=8^112^132^1$, we can choose 
\[
R_8 = [2,2,4], \quad  R_{12}= [3,3,6] \quad R_{32}=[4, 4,8,16].
\]

Let $t$ be a part of $T$ and let $R_t=[r_1,r_2,\dots, r_\ell]$ be the corresponding parts of $R$, listed in ascending order. Then $t/r_1=2^d$, for some positive integer $d$. Let $s_i=r_i/r_1$, for $i=1,2,\dots,\ell$, so $s_1+s_2+\dots+s_\ell=2^d$. Each integer $s_i$ is a power of $2$. Hence any element of the symmetric group $S_{2^d}$ with cycle type $[s_1,s_2,\dots,s_\ell]$ has order a power of 2. It follows that such an element is contained in a Sylow $2$-subgroup $P_{2, 2^d}$. Consequently, the group   $C_{r_1}\dot{\times} P_{2, 2^d}$ contains an element of cycle type $R_t$, where $C_{r_1}$ is generated by an $r_1$-cycle. 

Let $r_1=p_1^{a_1}p_2^{a_2}\cdots p_k^{a_k}$ be the prime decomposition of $r_1$ and, as usual, set $q_i=p_i^{a_i}$, for $i=1,2,\dots, k$. We observed at the start of the proof that $C_{r_1}$ is a subgroup of 
$P_{p_1, q_1}\dot{\times} P_{p_2, q_2}\dot{\times}\cdots\dot{\times}P_{p_\ell, q_\ell}$, so the group 
\[
 P_{p_1, q_1}\dot{\times} P_{p_2, q_2}\dot{\times}\cdots\dot{\times}P_{p_\ell, q_\ell}\dot\times P_{2, 2^d}
\]
contains an element of cycle type $R_t$. If $p_1,p_2,\dots, p_\ell$ are all odd, then we call this group $N_1$. If this is not the case, then we relabel so that $p_\ell=2$ and we make this group bigger, by defining
\[
 N_1=P_{p_1, q_1}\dot{\times} P_{p_2, q_2}\dot{\times}\cdots\dot{\times}P_{p_{\ell-1}, q_{\ell-1}}\dot\times P_{2, q_\ell2^d}.
\]
Observe that, again, $N_1$ contains an element of cycle type $R_t$. 

If $T$ has $e$ parts, then we repeat this process, and obtain groups $N_1,N_2,\dots, N_e$. Notice that $N_1\times N_2\times \cdots \times N_e$ embeds in $S_n$ (naturally and intransitively), and observe that it lies in $N(T)$ and contains an element of cycle type $R$, as required.
\end{proof}

We can now prove the results stated in the introduction.

\begin{proof}[Proof of Theorem~\ref{thm1}]
Proposition~\ref{p: distinct partition} tells us that if the cycle type of an element $g$ of $S_n$ is a distinct partition of $n$, then $g$ lies within a unique maximal nilpotent subgroup of $S_n$. Recall from the introduction that we denote the collection of all such maximal nilpotent subgroups by $\mathcal{M}$. Thus $\mathcal{M}=\{X:X\in N(T)\text{ and } T\in\DP(n)\}$. By Proposition~\ref{p: cover}, $\mathcal{M}$ is a cover of $S_n$. Furthermore, the uniqueness property of Proposition~\ref{p: distinct partition} implies that $\mathcal{M}$ is the unique minimal nilpotent cover of $S_n$ by maximal nilpotent subgroups. This concludes the proof of Theorem~\ref{thm1}.
\end{proof}

Corollary~\ref{cor1} follows immediately from Proposition~\ref{p: distinct partition} and Theorem~\ref{thm1}.

\begin{proof}[Proof of Corollary~\ref{cor2}]
First we establish that $\Sigma_N(S_n)= \sigma_N(S_n)$, observing that we already know that $\sigma_N(S_n)\leq \Sigma_N(S_n)$ (which, as noted in the introduction, is true more generally). For the reverse inequality, given $N\in \mathcal{M}$ we can find an element $g_N$ (with cycle type a distinct partition of $n$) for which $N$ is the unique maximal nilpotent subgroup of $S_n$ containing $g_N$. The set $\{g_N : N\in\mathcal{M}\}$ is a non-nilpotent subset of $S_n$ of size  $|\mathcal{M}|$, so $\Sigma_N(S_n)\leq \sigma_N(S_n)$, as required.

Next, observe that 
\[
|\mathcal{M}| = \sum_{T\in\DP(n)}|N(T)|,
\]
where $|N(T)|$ is equal to the index of the normalizer $\NORM_{S_n}(N_g)$ of $N_g$ in $S_n$, for any permutation $g$ of cycle type $T$. Let $O_1,O_2,\dots,O_k$ be the orbits in $\{1,2,\dots,n\}$ of $g$. Using Proposition~\ref{p: distinct partition} we can write $N_g$ as a direct product $N_{1}\times N_{2} \times \dots \times N_{k}$, where $N_i$ is a subgroup of $\Symm(O_i)$, for $i=1,2,\dots,k$. Now, if $h\in \NORM_{S_n}(N_g)$, then $h$ must permute the orbits $O_i$, and since they are of distinct orders we see that $h$ fixes each orbit. Consequently, the normalizer $\NORM_{S_n}(N_g)$ is the direct product of the normalizers of the subgroups $N_i$ in $\Symm(O_i)$, for $i=1,2,\dots,k$.

Let $N$ be any one of the subgroups $N_i$ and let $t=|O_i|$; thus $t$ is one of the parts of $T$. We write $\PR(t)=[p_1,p_2,\dots, p_\ell]$ and $\PP(t)=[q_1,q_2,\dots, q_\ell]$; then Proposition~\ref{p: distinct partition} tells us that
\[
N=P_{p_{1}, q_{1}}\dot{\times} P_{p_{2},q_{2}}\dot{\times} \dotsb \dot{\times} P_{p_{\ell}, q_{\ell}},
\]
for Sylow $p_{j}$-subgroups $P_{p_{j},q_{j}}$ of $S_{q_{j}}$, $j=1,2,\dots,\ell$. Now, if $h$ belongs to the normalizer of $N$ (in $\Symm(O_i)$), then, for each $j=1,2,\dots,\ell$, the permutation $h$ must preserve the unique system of imprimitivity of $N$ comprising $t/q_j$ sets of size $q_j$. Consequently, the normalizer of $N$ is the direct product of the normalizers $\NORM_{S_{q_{j}}}(P_{p_{j},q_{j}})$, for $j=1,2,\dots,\ell$. The normalizer $\NORM_{S_{q_{j}}}(P_{p_{j},q_{j}})$ is known to have order $(p_j-1)\smallstrut^{a_j}p_i^{e_j}$, where $q_j=p_j^{a_j}$ and $e_i=(p_i^{a_i}-1)/(p_i-1)$. Hence
 \[
|N(T)|=  \frac{|S_n|}{\lvert\NORM_{S_n}(N_g)\rvert} =
\frac{n!}{\prod\limits_{t\in T}\prod\limits_{i=1}^\ell (p_i-1)\smallstrut^{a_i}p_i^{e_i}},
\]
as required. 
\end{proof}

\section{Nilpotent covers of alternating groups}\label{section nilpotent  covers}

In this section we show that nilpotent covers of alternating groups do not share the properties of nilpotent covers of symmetric groups that we have established. We denote the alternating group on $n$ letters by~$A_n$. The following theorem contrasts with Corollaries~\ref{cor1} and~\ref{cor2}.

\begin{thm}\label{t: an}\leavevmode
\begin{enumerate}
 \item\label{ani}  For $n=9,12,14$ or $n\geq 16$, no minimal cover of $A_n$ by maximal nilpotent subgroups is~normal.
 \item\label{anii} For $n=9, 20,22,24,26,28,30$ or $n\geq 32$, we have $\sigma_N(A_n)< \Sigma_N(A_n)$.
\end{enumerate}
\end{thm}

The picture is mixed for the remaining values of $n$; for example, $\sigma_N(A_n)=\Sigma_N(A_n)$ when $n=3,4,\dots,8$, and there is a unique minimal cover of $A_n$ by maximal nilpotent subgroups for some values of $n<32$ and not others.

To prove Theorem~\ref{t: an} we need the following stronger version of Lemma~\ref{l: intransitive dp}.

\begin{lem}\label{l: intransitive dp 2}
Let $O_1,O_2,\dots,O_k$	be the orbits in $\{1,2,\dots,n\}$ of an element $g$ of $S_n$, and let $t_i=|O_i|$, for $i=1,2,\dots, k$. Suppose that $t_1=t_2$ and $t_2,t_3,\dots,t_n$ are distinct with none of them equal to $2t_1$. Then any nilpotent subgroup $N$ of $S_n$ containing $g$ satisfies
\[
N \leq \Symm(O_1\cup O_2)\times \Symm(O_3) \times \dots \times \Symm(O_k).
\]
\end{lem}
\begin{proof}
Let $P$ be any Sylow $p$-subgroup of $N$ and let $h\in P$. It suffices to prove that $(O_1\cup O_2)^h=O_1\cup O_2$ and $O_i^h=O_i$, for $i=3,4,\dots, k$.  We denote by $m$ the order of $g$ and let $d=|m|_{p'}$.  The argument now follows that of the proof of Lemma~\ref{l: intransitive dp} from the second paragraph to the end of the penultimate paragraph, so we proceed from there.
 
Observe that each of the integers $t_1/d,t_2/d,\dots,t_\ell/d$ is a power of $p$. Further, with the  exception of $t_1/d$ and $t_2/d$, these integers are distinct. It follows that the $p$-ary decomposition of $n/d$ can be written as 
\[
n/d=2 t_1/d + t_3/d +t_4/d+\dots + t_\ell/d.
\]
Since $t_3/d,t_4/d,\dots,t_\ell/d$ are distinct $p$th powers, none equal to $t_1/d$ or $2t_1/d$, we can argue as in Lemma~\ref{l: intransitive dp}, using the observations about orbits of Sylow $p$-subgroups stated before that lemma, to see  that $O_i^h=O_i$, for $i=3,4,\dots,k$. Clearly then $(O_1\cup O_2)^h=O_1\cup O_2$, as required.
\end{proof}

We recall the partition notation introduced in Proposition~\ref{p: cover}, in which $h\in S_n$ is said to have cycle type $\lambda_1^{a_1}\lambda_2^{a_2}\dotsb \lambda_k^{a_k}$ if $h$ is  a product of $a_i$ disjoint $\lambda_i$-cycles, for $i=1,2,\dots, k$, and $a_1\lambda_1+a_2\lambda_2+\dots +a_k\lambda_k=n$.

\begin{proof}[Proof of Theorem~\ref{t: an}]
First we prove assertion~\ref{ani}. To begin we will need two easy facts about elements in $A_9$. 
\begin{enumerate}
 \item[(a)] If $g$ is an element of cycle type $4^21^1$ in $A_9$, then $\CENT_{S_9}(g)$ is a $2$-group; consequently, a maximal nilpotent subgroup of $S_9$ containing $g$ is a Sylow $2$-subgroup of $S_9$ (and a maximal nilpotent subgroup of $A_9$ containing $g$ is a Sylow $2$-subgroup of $A_9$). 
 \item[(b)] Every element in a Sylow $2$-subgroup of $A_9$ lies in at least three Sylow $2$-subgroups of $A_9$.
\end{enumerate}

It should be clear that these facts yields assertion~\ref{ani} for $n=9$. Indeed, Fact (a) implies that a cover of $A_9$ by maximal nilpotent subgroups must contain some Sylow $2$-subgroups of $A_9$, but Fact (b) implies that the union of all but two of the Sylow $2$-subgroups of $A_9$ is equal to the union of all Sylow $2$-subgroups of $A_9$. Thus a minimal cover by maximal nilpotent subgroups must contain some but not all Sylow $2$-subgroups of $A_9$.

Next we use these facts about $A_9$ to deal with larger alternating groups. 

Suppose that $n\geq 17$ and $n$ is odd. Let $g$ be an element of $A_n$ with cycle type $4^21^13^1(n-12)^1$. Lemma~\ref{l: intransitive dp 2} implies that a maximal nilpotent subgroup $N$ of $A_n$ that contains $g$ must be a subgroup of $S_{9}\times S_{n-9}$. Indeed, Fact (a) implies that $N\leq P\times S_{n-9}$, where $P$ is a Sylow $2$-subgroup of $S_{9}$. But now Fact (b) implies that we can select two conjugates of $N$ in $A_n$ such that the the union of all of the conjugates of $N$ but these two is equal to the union of all of the conjugates of $N$. Thus if $\mathcal{M}$ is a minimal cover by maximal nilpotent subgroups, and $N$ is an element of $\mathcal{M}$ that contains $g$, then $\mathcal{M}$ must contain some but not all conjugates of $N$ in $A_n$.

A similar argument can be applied when $n\geq 12$ and $n$ is even. In this case we choose an element $g$ of $A_n$ with cycle type $4^21^1(n-9)^1$. This completes the proof of assertion~\ref{ani}.

For assertion~\ref{anii} we begin by proving the result for $n=9$. Consider first elements of the following cycle types:
\[
 9^1, \quad 6^12^11^1, \quad 5^13^11^1, \quad 4^13^12^1.
\]
Proposition~\ref{p: distinct partition} implies that each such element lies in a unique maximal nilpotent subgroup of $S_9$, so each one lies in a unique maximal nilpotent subgroup of $A_9$. The same proposition tells us the structure of the associated maximal nilpotent subgroups, and we can calculate directly that, between them, these subgroups contain all elements of $A_9$ except those of cycle types
\[
 7^11^2, \quad 5^12^2, \quad 4^21^1.
\]
It is easy to check that any element of either of the first two types lies in a unique maximal nilpotent subgroup of $A_9$. Thus we conclude that there is a unique cover by maximal nilpotent subgroups of $A_9\setminus \mathcal{C}$, where $\mathcal{C}$ is the conjugacy class of elements of type $4^21^1$. What is more, each nilpotent group $N$ in this cover contains an element $g_N$ with the property that $N$ is the only maximal nilpotent subgroup of $A_9$ that contains $g_N$.

Let $\mathcal{M}$ be a minimal cover of $\mathcal{C}$ by Sylow $2$-subgroups of $A_9$. For $\alpha\in\{1,2,\dots, 9\}$, let $\mathcal{C}_\alpha$ be the set of elements in $\mathcal{C}$ that fix the element $\alpha$. Observe that a Sylow 2-subgroup of $A_9$ that contains an element of $\mathcal{C}_\alpha$ must itself fix $\alpha$. It is therefore clear that there is a unique minimal set  $\mathcal{M}_\alpha\subset\mathcal{M}$ that covers $\mathcal{C}_\alpha$, and this set $\mathcal{M}_\alpha$ is simply the Sylow $2$-subgroups of $\mathcal{M}$ that fix the element $\alpha$.  

On the other hand, let $X$ be a non-nilpotent subset of $\mathcal{C}$. Note that if $g,h\in \mathcal{C}$ have different fixed points then $\langle g, h\rangle$ is not nilpotent (to see this, simply consider orbit sizes). This implies that if $X_\alpha$ is the set of elements in $X$ that fix a particular element $\alpha\in\{1,2,\dots,9\}$, then $X_\alpha$ is a maximal non-nilpotent subset of $\mathcal{C}_\alpha$.

Now, the elements of $\mathcal{C}_\alpha$ form a conjugacy class in the subgroup of $A_9$ that is isomorphic to $A_8$ and stabilizes $\alpha$; likewise the elements of $\mathcal{M}_\alpha$ are Sylow $2$-subgroups of this same copy of $A_8$. Hence to show that $\sigma_N(A_9)<\Sigma_N(A_9)$ it is sufficient to prove the following claim.

\begin{claim}
	Let $\mathcal{C}_1$ be the conjugacy class of elements of cycle type $4^2$ in $A_8$. The size of a maximal non-nilpotent subset of $\mathcal{C}_1$ is strictly less than the number of Sylow $2$-subgroups of $A_8$ required to cover $\mathcal{C}_1$.
\end{claim}

\begin{claimproof}
	Our proof proceeds with help from {\tt GAP} as follows.
	\begin{enumerate}
		\item Construct a graph $\Gamma$ as follows:
		\begin{itemize}
			\item The vertices of $\Gamma$ are the 630 cyclic subgroups of $A_8$ that are generated by an element of cycle type~$4^2$.
			\item Two vertices $X$ and $Y$ are joined by an edge in $\Gamma$ if $\langle X, Y\rangle$ is nilpotent.
		\end{itemize}
		\item One can check directly that $\Gamma$ is a vertex-transitive graph and that every vertex of $\Gamma$ has degree $14$.
		\item A Sylow $2$-subgroup of $A_8$ contains $14$ vertices of $\Gamma$. All of these vertices are connected to each other; thus the clique number (the maximal size of a complete subgraph) of $\Gamma$ is at least $14$.
		\item The clique--coclique bound for vertex-transitive graphs tells us that the coclique number of $\Gamma$ is at most $630/14=45$. Here the coclique number is the maximum size of an independent set in $\Gamma$; equivalently, it is the maximum size of a non-nilpotent subset of $\mathcal{C}_1$. Thus to prove the claim we must show that the number of Sylow $2$-subgroups of $A_9$ required to cover $\mathcal{C}_1$ is strictly greater than $45$.
		\item Since each Sylow $2$-subgroup of $A_8$ contains $14$ vertices of $\Gamma$, it is clear that such a cover will require at least $630/14=45$ Sylow $2$-subgroups. If a cover of size $45$ exists, then it must be \emph{disjoint}; that is, if $P_1$ and $P_2$ are two of the $45$ Sylow $2$-subgroups in such a cover, then $P_1\cap P_2 \cap \mathcal{C}_1=\varnothing$. Let us start to construct such a cover $\mathcal{N}$.
		\item Since the set of Sylow $2$-subgroups forms a single conjugacy class of subgroups in $A_8$, we can pick any of the $315$ Sylow $2$-subgroups $P_1$ as our first element of the cover. To choose our second element we must eliminate all Sylow $2$-subgroups $P_2$ for which $P_1\cap P_2 \cap \mathcal{C}_1\neq\varnothing$; {\tt GAP} tells us that we have 276 to choose from.
		\item We now use {\tt GAP} to establish the existence of an element $g_0\in\mathcal{C}_1\setminus P_1$ that lies in exactly $2$ of these remaining 276 Sylow $2$-subgroups. 
		\item For at least one choice of $g_0$, the following holds: let $P_2$ be either of the two Sylow $2$-subgroups that contain $g_0$ and satisfy $P_1\cap P_2 \cap \mathcal{C}_1$ is empty. Thus the cover $\mathcal{N}$ must contain such a subgroup $P_2$ as its second element. To choose our third element we must eliminate all Sylow $2$-subgroups $P_3$ for which $P_i\cap P_3 \cap \mathcal{C}_1\neq\varnothing$ for $i=1$ or $2$; {\tt GAP} tells us that, for either choice of $P_2$, we have 247 to choose from. Furthermore, for either choice, we use {\tt GAP} to establish that these 247 Sylow $2$-subgroups of $A_8$ contain precisely 600 vertices of $\Gamma$. However, $P_1\cup P_2$ contains 28 vertices of $\Gamma$, so we conclude that the elements of $\mathcal{N}$ can contain at most 628 vertices of $\Gamma$. Hence $\mathcal{N}$ does not cover $\Gamma$, so contrary to the assumption in part~(v), $\mathcal{N}$ must have order strictly greater than 45, as required. \hfill $\blacksquare$
	\end{enumerate}
\end{claimproof}

Let us now prove assertion~\ref{anii} of Theorem~\ref{t: an} for $n>9$. First, if $n\geq 20$ and $n$ is even, then Lemma~\ref{l: intransitive dp 2} implies that if $N$ is a nilpotent subgroup of $S_n$ containing an element of one of the types
\begin{center}
\begin{tabular}{m{3cm}m{3cm}m{3cm}m{3cm}}
 $9^1(n-9)^1$,  & $6^12^11^1(n-9)^1$, &$5^13^11^1(n-9)^1$,  &$4^13^12^1(n-9)^1$,\\
 $7^11^2(n-9)^1$, &$5^12^2(n-9)^1$,  &$4^21^1(n-9)^1$, &
\end{tabular}
\end{center}
then $N$ lies in a maximal subgroup of $S_n$ isomorphic to $S_9\times S_{n-9}$. Indeed, if $g$ is any such element for which the $(n-9)$-cycle is $(10,11,\dots, n)$, and $N$ is a nilpotent subgroup containing $g$, then $N$ is a subgroup of $M=A_9\times \langle (10,11,\dots, n)\rangle$. In particular, any cover of $A_n$ by maximal nilpotent subgroups must have a subset that is a cover of $M$ by maximal nilpotent subgroups of $M$. But elements of such a cover have the form $N_0 \times \langle (10,11,\dots, n)\rangle$, where $N_0$ ranges over a cover of $A_9$ by maximal nilpotent subgroups. We have seen that it is not possible to construct such a cover with the property that we can pick an element $g_0$ in each member $N_0$ of the cover such that the resulting set of elements $g_0$  is a non-nilpotent subset of $A_n$. In other words $\sigma_N(M)<\Sigma_N(M)$, and consequently $\sigma_N(A_n)<\Sigma_N(A_n)$.

A similar argument can be applied when $n\geq 33$ and $n$ is odd, but with the parts $11^1(n-20)^1$ in place of~$(n-9)^1$.
\end{proof}

Theorem~\ref{t: an} suggests that it is unlikely that there exists a closed formula for $\Sigma_N(A_n)$ or $\sigma_N(A_n)$ in the vein of Corollary~\ref{cor2}. The formula in that corollary is obtained by summing over certain partitions, each corresponding to a conjugacy class of maximal nilpotent subgroup, the members of which lie in a minimal cover. Theorem~\ref{t: an} implies that a minimal cover of $A_n$ by nilpotent subgroups is not a union of conjugacy classes of subgroups, hence this approach seems futile.

For instance, to calculate $\Sigma_N(A_{16})$ one would need to establish how many of the 638\,512\,875 Sylow $2$-subgroups of $A_{16}$ are required to cover the conjugacy class of cycle type $8^2$ elements. Similarly, referring to the proof above, to calculate $\sigma_N(A_9)$ one would need to establish the maximum size of a non-nilpotent subset of the conjugacy class of $4^2$ elements in $A_8$; the proof above gives an upper bound of $45$ while computer calculations have established a lower bound of $39$. The exact value requires more work.

\section{Nilpotent and abelian covers of other groups}\label{section other covers}

In this final section we contrast our work with known results on covers of almost simple groups by considering the following question.

\begin{question}\label{q four}
	Let $G$ be an almost simple group. Which of the following properties are satisfied by $G$?
	\begin{enumerate}
		\item There is a unique minimal cover of $G$ by maximal nilpotent subgroups.
		\item Each minimal cover of $G$ by maximal nilpotent subgroups is normal.
		\item There exists a minimal cover of $G$ by maximal nilpotent subgroups that is normal.
		\item $\Sigma_N(G)=\sigma_N(G)$. 
	\end{enumerate}
\end{question}

Evidently, if $G$ satisfies (i) then it satisfies (ii), and if $G$  satisfies (ii) then it satisfies (iii). We are unaware of further dependencies between the four statements.

We have seen that all four properties hold for $S_n$ (when $n\geq 5$) and all four properties fail for $A_n$ (when $n\geq 32$). Results in the literature \cite{Azizollah} (following earlier work of Azad \cite{azad}) confirm that the fourth property holds when $G$ is a finite group of Lie type of rank 1; in fact, the proofs given in \cite{Azizollah} confirm that if $G$ is ${\rm PGL}_2(q)$ or ${\rm PSL}_2(q)$, then all four properties hold.

One can frame analogous versions of Question~\ref{q four} with abelian covers or solvable covers (say) in place of nilpotent covers. Apparently there is no literature on solvable covers, but abelian covers of $S_n$ have been studied before. We summarise the findings here, since they contrast with the results for nilpotent covers of $S_n$. 

An \emph{abelian cover} of a group $G$ is a finite family of abelian subgroups of $G$ whose union is equal to $G$. We write $\Sigma_A(G)$ for the smallest possible size of an abelian cover of $G$. A \emph{non-commuting subset} of a group $G$ is a subset $X$ of $G$ such that any two distinct elements of $X$ generate a non-abelian group (or, equivalently, any two distinct elements of $X$ do not commute). We write $\sigma_A(G)$ for the size of the largest non-commuting subset of $G$.

 Just as for nilpotent covers, it is clear that $\sigma_A(G)\leq \Sigma_A(G)$. However, in contrast to Corollary~\ref{cor2}, Brown \cite{brown1, brown2} proved that $\sigma_A(S_n)< \Sigma_A(S_n)$ for $n\geq 15$. Furthermore, Barrantes {\it et al.} showed that, likewise, $\sigma_A(A_n)< \Sigma_A(A_n)$ for $n\geq 20$ \cite{BGR}. It would be of interest to determine whether the minimal covers of $S_n$ and $A_n$ by maximal abelian subgroups are unique or normal.

\subsection*{Acknowledgments}

The main result of this paper is based on work from the PhD thesis of the second author, supervised by the first and third authors. This supervisory arrangement has been made possible through a grant from the Mentoring African Research in Mathematics scheme sponsored by the LMS, IMU and AMMSI. All three authors thank these organisations for their invaluable support. 

We would also like to thank the anonymous referee for helpful suggestions and the editors of JGT for suggesting we extend our treatment to include the alternating groups.

\bibliographystyle{plain}
\bibliography{references-2}
\end{document}